\documentclass[12pt]{amsart}
\usepackage[a-2u]{pdfx}

\usepackage[utf8]{inputenc}
\usepackage{amsmath}        
\usepackage{amsfonts}       
\usepackage{amsthm}         
\usepackage{amssymb} 
\usepackage{bm}             

\usepackage{mathtools}
\usepackage{thmtools}
\usepackage{xcolor}
\usepackage{empheq}   
\usepackage{cleveref} 

\theoremstyle{plain}

\newtheorem{theorem}{Theorem}[section]

\newtheorem{lemma}{Lemma}[section]
\newtheorem{cor}{Corollary}[section]

\newcommand{\rr}{\mathbb{R}}
\newcommand{\cc}{\mathbb{C}}
\newcommand{\norm}[3]{\lVert #1 \rVert_{#2}^{#3}}
\newcommand*{\bu}{\bm u} 
\newcommand*{\bv}{\bm v}
\newcommand{\bvarphi}{\bm \varphi}
\newcommand{\entau}[1]{[{#1}\bm n]_{\tau}}
\newcommand{\diver}{\operatorname{div}}
\newcommand{\beh}{\bm h}
\newcommand{\bef}{\bm f}
\newcommand{\Bdry}{\Gamma}
\newcommand{\Du}{{\rm D}\bm u}

\newcommand{\bpjx}{\mathcal J_1^{x'}}

\newcommand{\hhinf}{\mathcal{H}^{\infty}}

\title[Stokes with DBC]{On $L^p$-semigroup to Stokes equation with dynamic slip boundary condition\\ in the half-space}

\thanks{Authors acknowledge the support of the project No. 20-11027X financed by the Czech Science Foundation (GA\v{C}R)}

\thanks{* Corresponding author.}

\author[D.~Pra\v{z}\'{a}k]{Dalibor Pra\v{z}\'{a}k$^*$}
\address{Charles University, Faculty of Mathematics and Physics, Department of Mathematical Analysis, Sokolovsk\'{a}~83, 186~75~Prague~8, Czech Republic}
\email{prazak@karlin.mff.cuni.cz}

\author[M. Zelina]{Michael Zelina$^*$}
\address{Charles University, Faculty of Mathematics and Physics, Department of Mathematical Analysis, Sokolovsk\'{a}~83, 186~75~Prague~8, Czech Republic}
\address{University of Uppsala, Department of Mathematics, Ångströmlaboratoriet, Regementsvägen 10, 751 06 Uppsala, Sweden}
\email{michael.zelina@math.uu.se}


\date{}

\begin{document}

\begin{abstract}
We consider evolutionary Stokes system, coupled with the
so-called dynamic slip boundary condition, in the simple
geometry of a $d$-dimensional half-space. Using the standard
technique of the Fourier transform in tangential directions, we obtain
an explicit formula for the resolvent. We then deduce
estimates for both the weak (i.e. $W^{1,p}$) and strong (hence
$W^{2,p}$) solutions, which are optimal in terms of the data
belonging to appropriate negative Sobolev or fractional Besov
space. In the latter case $L^p$-integrability of the pressure
gradient is included.  We allow for solutions with
non-zero divergence, thus preparing the way for extensions to
general domains. As a by-product, we show that the system
generates an analytic semigroup in $L^p(\Omega)\times
L^p(\partial \Omega)$.

Our approach remains elementary in the sense that only the
classical Mikhlin multiplier theorem will be used. The methods
of $\mathcal{H}^{\infty}$-calculus are implicitly present; but we stay
away from the concept of $R$-boundedness and related heavy
functional analytic machinery.
\end{abstract}

\maketitle

\section{Introduction}	\label{S:1}

This paper is motivated by the system
\begin{equation}	\label{evo1}
\begin{aligned}	
\partial_t \bu - \diver 2\nu \Du + \nabla \pi & = \bef,
\qquad \diver \bu = 0
\quad \textrm{ in $(0, T) \times \Omega$} ,
\\		
\beta \partial_t \bu + \alpha \bu + 2\nu \entau{(\Du)} &= \beta \beh,
\qquad \bu \cdot \bm n = 0
\quad \textrm{ on $(0, T) \times \Bdry$} ,
\end{aligned}
\end{equation}
where $\alpha$, $\beta$, and $\nu$ are positive parameters, $\Omega = \{ x_d > 0 \}$ is the $d$-dimensional
half-space, and $\Bdry = \{ x_d = 0 \}$ the corresponding
boundary. Furthermore, $\bm n$ is the outer normal vector, and
$\tau : \rr^{d} \to \rr^{d-1}$ denotes the tangential component, i.e.
\begin{equation*}
[(\Du) \bm n]_{\tau} = (\Du)\bm n  - ((\Du) \bm n  \cdot \bm n) \bm n .
\end{equation*}

Note that \eqref{evo1} consists of two evolution
equations, coupled via the trace operator. The problem can be
efficiently treated in the Hilbert $L^2$-setting, since there
is an underlying symmetric operator $A:V \to V^*$, where $V$ is
a suitable subspace to $W^{1,2}(\Omega)$, see e.g.
\cite{ABM21}, cf.\ also the weak formulation \eqref{res:weak}
below. However, the $L^p$-theory is expectedly more difficult
if $p\neq 2$ and up to our knowledge, the problem \eqref{evo1}
has not been treated in the literature (although a number of
closely related results are available, see also the discussion
below). In the present paper, we intend to contribute to this,
with two main results being: the existence of an analytic
semigroup in $L^p(\Omega) \times L^p(\Bdry)$, and the optimal
$W^{2,p}$ and $W^{1,p}$-regularity estimates for the
corresponding stationary (resolvent) problem. 

\par
Our paper is organized as follows. In Section~\ref{S:2}, we
introduce the function spaces and state the two main results:
\autoref{thm:Main1} and \autoref{thm:Main2}, providing
resolvent estimates of optimal regularity for both the weak
and strong solutions. In \autoref{thm:Cor1}, we show the
existence of an analytic semigroup; we also identify the domain
of its generator.  Main technical results of the paper 
follow: in Section~\ref{S:3}, we use (partial) Fourier
transform, in the spirit of \cite{FaSo94}, to express the
resolvent in terms of the boundary right-hand side. Key
estimates in the $L^p$ spaces are then obtained in
Section~\ref{S:4}. Our approach remains more elementary in that we
only rely on the classical Mikhlin multiplier theorem. One can, however,
note that \autoref{lem:mult2} in fact asserts the existence
of $\hhinf$-calculus for the operator $(-\Delta)^{1/2}$.
Finally, the proofs of the main theorems are
completed in Section~\ref{S:5}, essentially by combining the results of Section~\ref{S:4} with the well-known $L^p$-theory
for the (homogeneous) Dirichlet problem, as given e.g.
in \cite{Galdi11}.

\par
Although our results are formulated for the case of the half-space
$\Omega=\{x_d > 0\}$, an extension to more general domains is
straightforward, using standard localization and perturbation
arguments. Crucially, this requires extension of our estimates
to the case of functions $\bu$ with non-zero divergence. We
treat this situation explicitly in the part (iii) of
\autoref{thm:Main2}.

\par
Let us conclude the introduction by discussing some related,
both classical and more recent results. Concerning the
Stokes system subject to the homogeneous Dirichlet boundary condition,
the existence of an $L^p$-analytic semigroup was proven in \cite{Giga81}
via the theory of pseudo-differential operators. For a more elementary
approach based on the Fourier transform and classical multiplier
theorems, see \cite{FaSo94}. Recently, a different method using
the maximal operator was proposed in \cite{Shen12}; however, the
method seems to be limited to bounded domains.

\par
There is also an extensive literature devoted to problems similar
to \eqref{evo1}, but crucially, without the time derivative in the
boundary equation. For example, a fairly complete $L^p$-theory,
including the existence of an analytic semigroup, has been
established in \cite{AACG21}, see also \cite{Ghosh18} and the
references therein; yet the results again seem to be limited
to bounded $\Omega$ only.  In the case of a half-space, 
the problem was treated in \cite{Saal06}, albeit for $\beh=0$ only.
A more modern approach based on the $\hhinf$-calculus was employed here.
More recently, a closely related class of problems with dynamic boundary 
conditions was treated in \cite{BoKaKo17}, and maximal regularity results
were obtained.

\par
A rather comprehensive functional setup for treating resolvent problems
related to dynamic boundary conditions is described in \cite{Denk23};
see also the references therein. This method does provide sharp
regularity estimates in fractional Sobolev spaces, yet direct application
to the Stokes problem -- due to the incompressibility condition and the
pressure term -- does not seem viable. 

\par
Let us also mention recent paper \cite{BDK25}, where the existence
of an analytic semigroup for \eqref{evo1} is proven in the case of smooth
bounded domains, for the Hilbert setting of the maximal regularity
space $L^2(\Omega)\times W^{1/2,2}(\Gamma)$.

\par
Concerning possible applications, we remark that optimal regularity
estimates of the stationary Stokes problem (cf. \autoref{thm:Main2})
can be used to obtain higher regularity of evolutionary nonlinear problems,
via the standard bootstrap argument. See \cite{PZ24} for
the corresponding analysis in bounded domains, where the smoothness
of solutions was used to estimate the dimension of the global attractor,
in the case of the two-dimensional Navier-Stokes equations with dynamic boundary 
conditions.

\section{Notation and main results}	\label{S:2}

We set $\Omega = \{ (x',x_d) \in \rr^{d-1} \times \rr; x_d>0 \}$ and identify $\Bdry = \partial \Omega$ with $\rr^{d-1}$. Furthermore, $\bm n=(0,-1)$ is the outer normal, and
$\tau : \rr^{d} \to \rr^{d-1}$ denotes the tangential component.
For $p\in(1,\infty)$ and integer $k\ge1$, we denote 
by $L^p$ and $W^{k,p}$ the usual Lebesgue and Sobolev
spaces. Negative Sobolev space $W^{-1,p}$ is defined as the completion
of $L^q$ with respect to the norm
\[
	\norm{\bu}{W^{-1,p}(\Omega)}{} = \sup
	\Big\{ \int_\Omega \bu \cdot \bm \varphi \, {\rm d}x; \ \norm{\bm \varphi }{W^{1,p'}(\Omega)}{} \le 1
	\Big\} .
\]
It it the well-known fact (see e.g. \cite[Section 3.14]{AF03}) that $W^{-1,p}$ 
can be identified with the dual to $W^{1,p'}$.
\par
We also need Besov spaces $B^{s,p}_p$ with a general
real $s$; we allow for non-integer $s$ in the spirit of \cite{ADN59}.
The subscripts $\sigma$ or $n$ stand for the solenoidality, and 
zero normal component, respectively. Finally, we will extensively use 
the partial Fourier transform $\mathcal{F}_{x'\to\xi}$, which is defined by
\begin{align*}
\mathcal{F}_{x'\to\xi} \bu (x', x_d) = 
\widehat{\bu}(\xi, x_d) = \int_{\rr^{d-1}} \bu(x', x_d) 
		e^{-ix'\cdot \xi} \, {\rm d}x', \quad \xi \in \rr^{d-1}  .
\end{align*}

\par
Since it will not play any role in our paper, we set $\beta = \nu= 1$ in \eqref{evo1}.  We will be concerned with the related resolvent problem, i.e.
\begin{align}
(\lambda - \Delta) \bu + \nabla \pi &= \bef \quad \text{ in } \Omega , 
\label{Res:inside} \\
\diver \bu &= 0 \quad  \text{ in } \Omega , \label{Res:Div} \\
\bu \cdot \bm n &= 0 \quad  \text{ on } \Gamma , \label{Res:N} \\
(\lambda + \alpha)\bu + 2 [(\Du)\bm n]_{\tau} &= \beh \quad  \text{ on } \Gamma \label{Res:bound} ,
\end{align}
where $\lambda$ is a complex number. Note that $\diver 2 \Du = \Delta \bu$.
Multiplying \eqref{Res:inside} by a divergence-free, smooth, and compactly supported function
$\bvarphi$ with $\bvarphi \cdot \bm n = 0$ on $\Gamma$ and integrating
by parts, one obtains the integral formulation
\begin{equation}\label{res:weak}	
\begin{gathered} 	
\lambda \int_\Omega \bu \cdot \bvarphi \, {\rm d}x
+ (\lambda + \alpha) \int_\Bdry \bu \cdot \bvarphi \, {\rm d}S
+ \int_\Omega 2\Du : {\rm D}\bvarphi \, {\rm d}x
\\
=
 \int_\Omega \bef \cdot \bvarphi \,{\rm d}x
+  \int_\Bdry \beh \cdot \bvarphi \, {\rm d}S .
\end{gathered}
\end{equation}
This problem is easy to handle in the $L^2$-setting, using Korn's inequality, see \cite{AACG21} or \cite{ABM21}. However, our objective is to obtain
the $L^p$-theory, as stated in the following two main theorems.

\begin{theorem}	\label{thm:Main1}
Let $\alpha \ge0$, $p \in (1, + \infty)$, $\bef \in L^p_{\sigma,n}(\Omega)$ and $\beh \in L^p_{n}(\Bdry)$.
Then for any complex $\lambda$ with $\mathfrak{Re}\, \lambda > 0$ there exists 
a unique $\bu \in W^{1,p}_{\sigma,n}(\Omega)$ 
weak solution to \eqref{Res:inside}--\eqref{Res:bound}.
\par
If moreover $|\lambda| > \omega$ for some $\omega>0$ fixed, then
\begin{equation}	\label{est:Res}
\norm{\bu}{L^p(\Omega)}{} + \norm{\bu}{L^p(\Bdry)}{}
\le \frac{C_0}{|\lambda|} \big(
\norm{\bef}{L^p(\Omega)}{} + \norm{\beh}{L^p(\Bdry)}{} \big) ,
\end{equation}
where the constant $C_0$ depends on $d$, $p$ and $\omega$,
but is independent of $\lambda$ and $\alpha$.
\end{theorem}

\begin{theorem} \label{thm:Main2} Let $\alpha \ge 0$, $p\in (1, + \infty)$ 
and $\lambda \in \mathbb{C}$ with $\mathfrak{Re}\, \lambda >0$.
\begin{itemize}
    \item[(i)]   Let $\bef \in W^{-1,p}(\Omega)$, and let $\beh \in 
B^{-1/p,p}_{p;n}(\Bdry)$. Then one has the estimate
\begin{equation} 	\label{est:Reg1}
\norm{\bu}{W^{1,p}(\Omega)}{} 
	\le C_1 \big(
\norm{\bef}{W^{-1,p}(\Omega)}{} + \norm{\beh}{B^{-1/p,p}_p(\Bdry)}{} \big) .
\end{equation}
    \item[(ii)] Let $\bef \in L^p_{\sigma,n}(\Omega)$ and 
$\bm h \in B^{1-1/p,p}_{p;n}(\Bdry)$. Then the solution to resolvent
problem \eqref{Res:inside}--\eqref{Res:bound} is strong, and one has the estimate
\begin{align} \label{est:Reg2}
\norm{\bu}{W^{2,p}(\Omega)}{} + \norm{\nabla \pi}{L^p(\Omega)}{}
	& \le C_2 \big(
\norm{\bef}{L^p(\Omega)}{} + \norm{\beh}{B^{1-1/p,p}_p(\Bdry)}{}  \big) .
\end{align}
\item[(iii)] If \eqref{Res:Div} is generalized to $\diver \bu = g$, 
the estimates remain valid, with the norm of $g$ in $L^p(\Omega)$
or $W^{1,p}(\Omega)$ added to the right-hand side of \eqref{est:Reg1}
or \eqref{est:Reg2}, respectively.
\end{itemize}
The constants $C_1$, $C_2$ depend on $d$, $p$ and $|\lambda|$.
\end{theorem}

Postponing proofs of these theorems to Section~\ref{S:5},
we can now establish the existence of an analytic semigroup.

\begin{cor} \label{thm:Cor1}
System \eqref{evo1} generates an analytic semigroup 
in $L^p_{\sigma,n}(\Omega) \times L^p_n(\Gamma)$, where the domain of the
generator is contained in $B^{1+1/p-\varepsilon,p}_p(\Omega)$,
$\varepsilon>0$ arbitrary.
\end{cor}
\begin{proof}
Set $X = L^p_{\sigma,n}(\Omega) \times L^p_n(\Gamma)$, 
and define $A\bu = (\bef,\beh)$ by the (weak) formulation
of the resolvent, i.e. \eqref{res:weak}. We set
\[
	\mathcal{D}(A) = \{ \bu \in W^{1,p}_{\sigma,n}(\Omega) ;
	\ A\bu \in X \}
\]
-- with a slight abuse of notation, we identify the Sobolev
function $\bu$ with the couple $(\bu,\operatorname{tr} \bu)$.
It is easy to verify that $A$ is densely defined, and
the closedness follows from the resolvent estimate \eqref{est:Res};
see also \eqref{est:Res-w1} below.
By a standard result (see e.g. \cite[Corollary 3.7.14]{ABBN01}), 
the existence of semigroup follows.
\par
Let us identify the domain of $A$. In view of uniqueness
(cf.\ the argument after \eqref{est:Res-w1} below), we can treat
the regularity with respect to $\bef \in  L^p_{\sigma,n}(\Omega)$ and 
$\beh \in L^p_{n}(\Gamma)$ separately. If $\beh=0$, if follows from
\eqref{est:Reg2} that $\bu \in W^{2,p}(\Omega)$. However,
in contrast to the Dirichlet problem, this cannot be the domain since
\eqref{Res:bound} would then imply $\beh \in B^{1-1/p,p}_{p}(\Bdry)$.
\par
In fact, we claim that for any $\varepsilon > 0$
\begin{equation}	\label{domain-eps}
	B^{1+1/p+\varepsilon,p}_{p;\sigma,n}(\Omega) \subset
		\mathcal{D}(A)	\subset
	B^{1+1/p-\varepsilon,p}_{p;\sigma,n}(\Omega) .
\end{equation}
Indeed, interpolating \eqref{est:Reg1}--\eqref{est:Reg2} with
$\bef=0$, we get the second inclusion. On the other hand,
if $\bu \in B^{1+1/p+\varepsilon,p}_p(\Omega)$, then
the left-hand side of \eqref{Res:bound} belongs to 
$B^{\varepsilon,p}_p(\Gamma)$, in contradiction to
the fact that in general, $\beh \in L^p(\Gamma)$ only.
\par
Note that if $p=2$, one has the equality
$\mathcal{D}(A)=W^{3/2,2}_{\sigma, n}(\Omega)$,
as follows from \autoref{lem:fundH2} below.
\end{proof}

\section{Fundamental solution}	\label{S:3}
Here, we deal with \eqref{Res:inside}--\eqref{Res:bound} for $\bm f = 0$ and $\bm h = \bm \Phi = (\bm \phi , 0)$, i.e. we consider the system
\begin{align}
(\lambda - \Delta) \bu + \nabla \pi &= 0 \quad \text{ in } \Omega , \label{eq:InsideEquation} \\
\diver \bu &= 0 \quad  \text{ in } \Omega , \label{eq:Div} \\
\bu \cdot \bm n &= 0 \quad  \text{ on } \Gamma , \label{eq:Permeability} \\
(\lambda + \alpha)\bu + 2 [(\Du)\bm n]_{\tau} &= \bm \Phi \quad  \text{ on } \Gamma \label{eq:BoundaryEquation} .
\end{align}
We will (formally) solve the problem, using Fourier analysis.

\subsection{Reducing the equations}
By $\diver'$, $\Delta'$ we denote the corresponding operators with respect to $x'$. Let us now simplify \eqref{eq:Permeability}--\eqref{eq:BoundaryEquation}. First, $\bm n = ( 0, -1)$, and thus
\begin{equation*}
\bu \cdot \bm n = 0 \Leftrightarrow u_d = 0 \text{ on } \Gamma .
\end{equation*}
Hence, 
\begin{align*}
2(\Du)\bm n &=   
- \begin{pmatrix}
\partial_d  u_1 + \partial_1 u_d \\
\partial_d  u_2 + \partial_2 u_d \\
\vdots \\
\partial_d  u_{d-1} + \partial_{d-1} u_d \\
2 \partial_d  u_d
\end{pmatrix} 
=  
- \begin{pmatrix}
\partial_d  \bu'  \\
2 \partial_d  u_d
\end{pmatrix} 
\end{align*}
and
\begin{equation*}
2 [(\Du) \bm n]_{\tau} = 2(\Du)\bm n  - 2((\Du) \bm n  \cdot \bm n) \bm n = 
- \begin{pmatrix}
\partial_d  \bu'  \\
0
\end{pmatrix} . 
\end{equation*}
Therefore, we can rewrite \eqref{eq:BoundaryEquation} as follows
\begin{align}
(\lambda + \alpha) \bu' - \partial_d  \bu'  &= \bm \phi \quad \text{ on } \Gamma. \label{eq:BoundaryEquation2} 
\end{align}

Now, we aplly $\partial_d  \diver'$ to the first $(d-1)$ equations of \eqref{eq:InsideEquation} and $-\Delta'$ to the last one; adding these equations together we obtain
\begin{align*}
 (\lambda - \Delta) & \left[ \partial_d (\partial_1 u_1 + \partial_2 u_2   + \cdots  + \partial_{d-1} u_{d-1} ) - \Delta' u_d  \right]  \\
& \qquad \quad  + \left[ \partial_d (\partial_1^2 \pi + \partial_2^2 \pi  + \cdots  + \partial_{d-1}^2 \pi ) -  \Delta' \partial_d  \pi  \right] = 0 .
\end{align*}
The pressure term clearly vanishes and in the first bracket we use $\diver \bu = 0 \Leftrightarrow \diver' \bu' = -\partial_d  u_d  $ to get
\begin{align*}
- (\lambda - \Delta)\Delta u_d & = 0 \quad \text{ in } \Omega , \\
u_d &= 0 \quad \text{ on } \Gamma .
\end{align*}
Finally, we deal with the boundary condition \eqref{eq:BoundaryEquation2}. Applying $-\diver'$ yields
\begin{align*}
-(\lambda + \alpha)\, \diver'  \bu' +  \partial_d \diver' \bu' = -\diver' \bm \phi , 
\end{align*}
and thanks to $\diver' \bu' = -\partial_d  u_d$ we achieve
\begin{align*}
(\lambda + \alpha - \partial_d )  \partial_d u_d = -\diver' \bm \phi \quad \text{ on } \Gamma .
\end{align*}
Therefore, we reduced the whole problem to the question of solving the following system for $u_d$
\begin{align}
\begin{split}
- (\lambda - \Delta)\Delta u_d & = 0 \quad \text{ in } \Omega , \\
(\lambda + \alpha - \partial_d )  \partial_d u_d & = -\diver' \bm \phi \quad \text{ on } \Gamma , \\
u_d & = 0 \quad \text{ on } \Gamma .
\end{split} \label{eq:EquationForTheLastComponent}
\end{align}

\subsection{Fourier transform}
We solve \eqref{eq:EquationForTheLastComponent} by applying the Fourier transform $\mathcal{F}_{x' \to \xi}$. We need to deal with the system
\begin{align}
\begin{split}
(\lambda + |\xi|^2 - \partial_d ^2) (|\xi|^2 - \partial_d ^2)\widehat{u_d} &= 0 \quad \text{ for } x_d > 0 , \\
(\lambda + \alpha - \partial_d ) \partial_d  \widehat{u_d} &= -i \xi \cdot \widehat{\bm \phi} \quad \text{ for } x_d = 0 , \\
\widehat{u_d} (\xi , 0) &= 0 ,
\end{split}  \label{eq:FourierSystem}
\end{align}
for any $\xi \in \mathbb{R}^{d-1}$. Here, the characteristic roots are
\[ \pm \sqrt{\lambda + |\xi|^2} \quad \text{ and } \pm |\xi|  \,; \]
certainly, only the terms with the minus sign are relevant.
It will be useful to introduce the function
\begin{align}
m_0 (\lambda, \xi, x_d) = \frac{e^{-x_d \sqrt{\lambda + |\xi|^2} } - e^{-x_d |\xi|}}{\lambda + (\lambda + \alpha) ( \sqrt{\lambda + |\xi|^2} - |\xi|)} \, , \label{eq:Function}
\end{align}
which is nothing else than the fundamental solution to \eqref{eq:FourierSystem}, i.e. one has
 \begin{align*}
(\lambda + |\xi|^2 - \partial_d ^2) (|\xi|^2 - \partial_d ^2)m_0 &= 0 \quad \text{ for } x_d > 0 , \\
(\lambda + \alpha - \partial_d ) \partial_d  m_0 &= -1 \quad \text{ for } x_d = 0 , \\
m_0 (\lambda, \xi , 0) &= 0 .
\end{align*}
The solution for the last component is thus
\begin{align*}
\widehat{u_d}(\xi, x_d)  = i\xi \cdot m_0 (\lambda, \xi, x_d) \widehat{\bm \phi}(\xi)  .
\end{align*}

Now, we need to express  $\widehat{\pi}$ and $\widehat{\bm u'}$ using $m_0$. We start with the pressure. Let us apply $\diver'$ to the first $(d-1)$ equations of \eqref{eq:InsideEquation}; we get
\begin{align*}
\Delta' \pi = -\diver' (\lambda - \Delta) \bu' \quad \text{ in } \Omega .
\end{align*}
Because of $\diver' \bu' = -\partial_d  u_d$ we rewrite it in the form
\begin{align*}
\Delta' \pi =  (\lambda - \Delta) \partial_d  u_d \quad \text{ in } \Omega ,
\end{align*}
and use the Fourier transform to achieve
\begin{align*}
 \widehat{\pi}  =  -\frac{1}{|\xi|^2} (\lambda+|\xi|^2 - \partial_d^2)  \partial_d \widehat{u_d} .
\end{align*}
In other words
\begin{align*}
 \widehat{\pi} (\xi, x_d) 
 &=  -(\lambda+|\xi|^2 - \partial_d^2)  \partial_d m_0 (\lambda, \xi, x_d) \frac{ i\xi }{|\xi|^2}  \cdot \widehat{\bm \phi} (\xi) \\
 &= -\frac{\lambda e^{-x_d |\xi| } }{\lambda + (\lambda + \alpha) ( \sqrt{\lambda + |\xi|^2} - |\xi|)}  \frac{ i\xi }{|\xi|}  \cdot \widehat{\bm \phi} (\xi)  . 
\end{align*}

Up to now, everything was dimension-independent. At this point, 
we need to express the function $\bu'$. If $d = 2$, then everything is simpler. From \eqref{eq:Div} we easily find that
$$ \widehat{u_1}(\xi,x_d) = - \partial_{2} m_0 (\lambda, \xi, x_2) 
\widehat{\bm \phi}(\xi) . $$

In the general case $d > 2$ we start in the same way, i.e. we take the Fourier transform of \eqref{eq:Div} and express our unknown vector field 
$$ i \xi \cdot \widehat{\bu'} + \partial_d \widehat{u_d} = 0 \Leftrightarrow  \widehat{\bu'} =  \frac{i\xi }{|\xi|^2} \partial_d \widehat{u_d}  ,$$
which yields
\begin{align*}
 \widehat{\bu'} (\xi, x_d)  =-  \partial_d m_0 (\lambda, \xi, x_d) \frac{ \xi \otimes \xi }{|\xi|^2} \widehat{\bm \phi}(\xi)\, .
\end{align*}
This yields the solution to the Stokes system, but our condition \eqref{eq:BoundaryEquation} is not satisfied (in general). We need to add a correction term $\bm v = (\bm v', 0)$. To this purpose, consider the following problem
\begin{align*}
(\lambda - \Delta) \bv' & = 0 \quad \text{ in } \Omega , \\
\diver \bv' &= 0 \quad \text{ in } \Omega , \\
(\lambda + \alpha) \bv' - \partial_d \bv' &= \bm h' \quad \text{ on } \Gamma ,
\end{align*}
where 
$$ \widehat{\bm h'} = \widehat{\bm \phi} -  \frac{ \xi \otimes \xi }{|\xi|^2} \widehat{\bm \phi} .  $$

Now, by the Fourier transform we get
\begin{align*}
(\lambda + |\xi|^2 - \partial_d^2) \widehat{\bv'} & = 0 \quad \text{ in } \Omega , \\
(\lambda + \alpha - \partial_d) \widehat{ \bv'} &= \widehat{\bm h'} \quad \text{ on } \Gamma  .
\end{align*}
The characteristic roots are now $\pm \sqrt{\lambda + |\xi|^2}$, and thus
\[ \widehat{\bv'} (\xi, x_d) = \bm b(\xi) e^{-x_d \sqrt{\lambda + |\xi|^2}}\,.
\]
From the boundary condition we find
$$  \bm b(\xi) = \frac{1}{\lambda + \alpha + \sqrt{\lambda + |\xi|^2}} 
	\widehat{\bm h'}(\xi) . $$
Together,
\begin{align*}
 \widehat{\bu'} (\xi, x_d)  &= - \partial_d m_0 (\lambda, \xi, x_d) \frac{ \xi \otimes \xi }{|\xi|^2} \widehat{\bm \phi}(\xi) 
\\ &+ \frac{1}{\lambda + \alpha + \sqrt{\lambda + |\xi|^2}}\left(\text{Id}_{d-1}\, - \frac{\xi \otimes \xi }{|\xi|^2}\right) e^{-x_d \sqrt{\lambda + |\xi|^2}} \widehat{\bm \phi}(\xi) \, .
\end{align*}
To conclude this section, we summarize that for $d \geq 2$, 
the solution to \eqref{eq:InsideEquation}--\eqref{eq:BoundaryEquation}
can be expressed as 
\begin{equation} \label{eq:Multipliers} 
\begin{aligned}
 \widehat{\bu'} (\xi, x_d)  & =-  \partial_d m_0 (\lambda, \xi, x_d) \frac{ \xi \otimes \xi }{|\xi|^2} \widehat{\bm \phi}(\xi)  \\
 & + \frac{1}{\lambda + \alpha + \sqrt{\lambda + |\xi|^2}}\left(\text{Id}_{d-1}\, - \frac{ \xi \otimes \xi }{|\xi|^2}\right) e^{-x_d \sqrt{\lambda + |\xi|^2}} \widehat{\bm \phi}(\xi) , \\
\widehat{u_d}(\xi, x_d) & = i\xi \cdot m_0 (\lambda, \xi, x_d) \widehat{\bm \phi}(\xi) , \\
 \widehat{\pi} (\xi, x_d) &=  -\frac{\lambda e^{-x_d |\xi| } }{\lambda + (\lambda + \alpha) ( \sqrt{\lambda + |\xi|^2} - |\xi|)}  \frac{ i\xi }{|\xi|}  \cdot \widehat{\bm \phi} (\xi)  ,
\end{aligned}
\end{equation}
where $m_0$ was defined in \eqref{eq:Function}.
For $d = 2$, the second part of the first right-hand side vanishes.

\section{Multipliers and $L^p$-bounds}	\label{S:4}

The aim of this section is to establish $L^p$-estimates of the solution
obtained above. To begin with, we note that $m_0$ only depends on $z=|\xi|$. 
Setting also for simplicity $y = x_d$, with a slight abuse of notation we
can write
\begin{align}		\label{def-m0}
m_0 = m_0(\lambda,z,y) &= 
\frac{ \omega(\lambda,z) + z }{\alpha + \lambda + \omega(\lambda,z) + z} \cdot
\frac{ e^{-y \omega(\lambda,z)} - e^{-yz}}{\lambda} ,
\end{align}
where $\omega(\lambda,z) = \sqrt{\lambda + z^2}$. From relations \eqref{eq:Multipliers} it follows that
\begin{equation} \label{mult:simpler}
\begin{aligned}
 \widehat{\bu'} (\xi, y)  & = 
m_2 (\lambda,|\xi|, y) \frac{ \xi \otimes \xi }{|\xi|^2} \widehat{\bm \phi}(\xi) 
 + m_3(\lambda,|\xi|, y)  
\left(\text{Id}_{d-1} - \frac{ \xi \otimes \xi }{|\xi|^2}\right) 
\widehat{\bm \phi}(\xi) , 
\\
\widehat{u_d}(\xi, y) & = m_1(\lambda,|\xi|, y) \frac{i\xi}{|\xi|} \cdot
		\widehat{\bm \phi}(\xi) ,
\\
\widehat{\pi}(\xi,y) &= m_4(\lambda,|\xi|,y) \frac{i\xi}{|\xi|} \cdot
        \widehat{\bm \phi}(\xi) ,
\end{aligned}
\end{equation}
where we have set 
\begin{align}	
\begin{split}
	m_1(\lambda,z,y) &= z m_0(\lambda,z,y)	, 	\\
	m_2(\lambda,z,y) &= - \partial_y m_0(\lambda,z,y), \\
	m_3(\lambda,z,y) &= \frac{1}{\alpha + \lambda + \omega(\lambda,z)} \,
		e^{-y\omega(\lambda,z)} , \\
	m_4(\lambda,z,y) &= -\frac{ \omega(\lambda,z) + z }{\alpha + \lambda + \omega(\lambda,z) + z} \,
 e^{-yz} . \label{def-m1-m4}
\end{split}
\end{align}
As the mutlipliers $\xi/|\xi|$ and $\xi \otimes \xi / |\xi|^2 $
are easily seen to satisfy the Mikhlin condition \eqref{mich}, 
it is enough to treat just the terms $m_i$, $i=1,\dots,4$.  
Denote 
\begin{equation}	\label{df:sector}
\Sigma_{\theta} = \{ z \in \cc \setminus \{ 0 \} ; \
				|\arg z | < \theta \} ,
\end{equation}
Let $\hhinf(\Sigma_\theta)$ be the set of all bounded and analytic functions
in $\Sigma_\theta$.  We will use repeatedly the following estimate,
which follows easily from the  $\hhinf$-calculus for the operator
$(-\Delta)^{1/2}$ on $L^p$, see e.g. \cite{KuWe04}. For reader's
convenience, we provide a short proof based on the classical 
Mikhlin's theorem.
\begin{lemma}	\label{lem:mult2}
Let $\widehat{\bu}(\xi, y) = m(|\xi|, y) \widehat{\bm \phi}(\xi)$, 
where $m(\cdot, y) \in \hhinf(\Sigma_\theta)$ for each $y \ge 0$ fixed.
Set 
\begin{equation}
	k(y) = \sup_{ z\in \Sigma_\theta} |m(z,y)| \,,
	\qquad y \ge 0 \,.
\end{equation}
Then one has the estimate
\begin{equation}	\label{est-ky}
\norm{\bu}{L^p(\Gamma)}{} + 
\norm{\bu}{L^p(\Omega)}{} \le C_{p,d}\, 
	\big( k(0) + \norm{k}{L^p(0,\infty)}{} \big) 
							 \norm{\bm \phi}{L^p(\rr^{d-1})}{} .
\end{equation}
Furthermore, if under these assumptions $\widehat{\bu}(\xi, y) = m(|\xi|, y) \widehat{\bm v}(\xi,y)$
for any $y>0$ fixed, then
\begin{equation}	\label{est-ki}
	\norm{\bu}{L^p(\Omega)}{} \le C_{p,d} \norm{k}{L^{\infty}(0,\infty)}{}
			\norm{\bm v}{L^p(\Omega)}{} .
\end{equation}
\end{lemma}

\begin{proof}
Let $y \ge0$  be fixed. We want to verify the Mikhlin condition
\begin{equation}	\label{mich}
	|\xi|^k |\nabla^k m(|\xi|, y)| \le C_k (y), \qquad
		\xi \neq 0,\ k\ge0 ,
\end{equation}
see e.g. \cite[Theorem 5.2.7]{Graf01}. It follows once we prove 
\begin{equation*}	
	|s^k m^{(k)}(s, y) | \le C (y) \,,
		\qquad s \in (0,\infty) .
\end{equation*}
Let $s\in(0,\infty)$ be fixed, and let $\Gamma_s$ be the largest circle
centered in $s$, and contained in $\Sigma_\theta$. By Cauchy's
theorem
\[
s^k m^{(k)}(s, y) = \frac{1}{2\pi i} \int_{\Gamma_s}
\frac{s^k m(z, y)}{(z-s)^{k+1}}  \, {\rm d}z .
\]
Now it is easy to see that the right-hand side has a bound
only depending on $\theta$ and the supremum of $m(\cdot, y)$.

As the Mikhlin condition is satisfied, we see that
\[
	\norm{ \bu(\cdot, y)}{L^p(\rr^{d-1})}{}
		\le C_{p,d}\, k(y) \norm{\bm \phi}{L^p(\rr^{d-1})}{} 
\]
holds. The conclusion is immediate by Fubini's theorem.
\end{proof}

The following two lemmas are the main technical estimates
of the paper.

\begin{lemma}	\label{thm:Lemma1}
Let $\theta \in (0,\pi/4)$ be fixed. Then there exist constants $C$,
$c > 0$ and $\delta>0$ such that
\begin{align}	\label{est:omega}
	\mathfrak{Re}\, \omega(\lambda,z) &\ge c ( z + \sqrt{|\lambda|} ) ,
\\				\label{est:frac1}
	\left| \frac{z \left( e^{-y\omega(\lambda,z)} - e^{-yz}\right)}{\omega(\lambda,z) - z } \right| 
		&\le \frac{C e^{-\delta y |z|}}{1 + y\sqrt{|\lambda|}} \, ,
\\				\label{est:frac2}
	\frac{1}{| \alpha + \lambda + z + \omega(\lambda,z)|} 
		&\le \frac{C}{\alpha + |\lambda|}  \, ,
\\
	\label{est:frac3}
	\frac{1}{| \alpha + \lambda + z + \omega(\lambda,z)|} 
		&\le \frac{C}{\alpha + |z|^2}  \, ,		
\end{align}
for all $z\in\Sigma_\theta$, $y\ge0$ and $\lambda \in \Sigma_{\pi/2}$.
\end{lemma}

\begin{proof}
For \eqref{est:omega} see Lemma 5.2(b), (c) in \cite{Saal06}, and \eqref{est:frac1} follows from Lemma 5.3(b) therein. To show \eqref{est:frac2}, we proceed as follows.

Since $\lambda \in \Sigma_{\pi/2}$ and $z \in \Sigma_{\theta}$ for some $\theta < \pi/4$, we have $\lambda + z^2 \in \Sigma_{\pi/2}$. Hence $\omega (\lambda, z) \in \Sigma_{\pi/4}$ and $z + \omega (\lambda, z) \in \Sigma_{\pi/4}$. As $\frac{1}{\alpha + \lambda} \in \Sigma_{\pi/2}$, we get $\frac{z + \omega (\lambda, z)}{\alpha + \lambda} \in \Sigma_{3\pi/4}$, which means that this value is strictly away from $-1$, in particular
\begin{align*}
\Big|1 + \frac{z + \omega (\lambda, z)}{\alpha + \lambda} \Big| \geq C
\end{align*}
holds for some $C > 0$. Finally, 
\begin{align*}
\Big|\alpha + \lambda + z + \omega (\lambda, z) \Big| 
&\geq C |\alpha + \lambda| 
= C \sqrt{ (\alpha + \mathfrak{Re}\, \lambda)^2 + \mathfrak{Im}^2\, \lambda }  \\
&\geq C \sqrt{ \alpha^2 + |\lambda|^2  }
\geq C (\alpha + |\lambda|) ,
\end{align*}
where we used that $\alpha$ and $\mathfrak{Re}\, \lambda $ are positive.

Finally, as $z - z^2 + \lambda + \omega (\lambda, z) \in \Sigma_{\pi/2}$ and $\alpha + z^2 \in \Sigma_{2\theta}$, we have
\begin{align*}
\Big|1 + \frac{z - z^2 + \lambda + \omega (\lambda, z)}{\alpha + z^2} \Big| \geq C ,
\end{align*}
and the \eqref{est:frac3} follows as before.
\end{proof}

\begin{lemma}	\label{thm:Lemma2}
Let $\theta \in (0,\pi/4)$ be fixed. Then there exist constants $C > 0$
 and $\delta>0$ such that for all $z\in\Sigma_\theta$, $y\ge0$ 
and $\lambda \in \Sigma_{\pi/2}$
\begin{align}	\label{est:mi-first}
	|m_i(\lambda,z,y)| &\le 
\frac{C e^{-\delta y |z|}}{(1 + y\sqrt{|\lambda|})(\alpha+|\lambda|)} \, ,
\\				\label{est:mi-second}
	|z m_i(\lambda,z,y)| + |\partial_y m_i(\lambda,z,y)|
			&\le 
\frac{C e^{-\delta y |z|}}{(1 + y\sqrt{|\lambda|})\sqrt{\alpha+|\lambda|}} \, ,
\end{align}
for $i=1$, $2$, $3$ and
\begin{equation}	\label{est:m4}
	|m_4(\lambda,z,y)| 
	\le C e^{-\delta y |z|} \,.
\end{equation}
\end{lemma}

\begin{proof}
We will first prove the bounds for $m_i$ and $zm_i$, for all $i$. The bound for $\partial_y m_i$ will be a simple consequence of these. 

\underline{Ad $m_1$.} Using \eqref{est:frac1} and \eqref{est:frac2} we find
\begin{align*}
| m_1 (\lambda, z, y) | 
&= | z  m_0 (\lambda, z, y) |
 \leq \frac{C}{|\lambda|} \cdot \frac{e^{-\delta |z| y }}{1+y\sqrt{|\lambda|}}  \Big| \frac{(\omega (\lambda, z) - z)(\omega (\lambda, z)  + z)}{\alpha + \lambda + z +  \omega (\lambda, z)} \Big| \\
&\leq C  \frac{e^{-\delta |z| y }}{1+y\sqrt{|\lambda|}} \cdot  \frac{1}{ | \alpha + \lambda + z +  \omega (\lambda, z) | }  \\
&\leq   \frac{C}{ (1+y\sqrt{|\lambda|}) ( \alpha + |\lambda| )  } e^{-\delta |z| y} .
\end{align*}

Similarly, using \eqref{est:frac2} and \eqref{est:frac3} together, we get 
\begin{align*}
|  z  m_1 (\lambda, z, y) | 
& \leq C  \frac{e^{-\delta |z| y }}{1+y\sqrt{|\lambda|}} \cdot  \frac{|z| }{|\alpha + \lambda + z +  \omega (\lambda, z)|}  \\
& \leq C  \frac{e^{-\delta |z| y }}{1+y\sqrt{|\lambda|}} \cdot  \frac{1}{\sqrt{\alpha + |\lambda|}} \cdot \frac{|z|}{\sqrt{\alpha + |z|^2}} \,  ,
\end{align*}
from which the estimate for $z m_1$ follows.

\underline{Ad $m_3$.} Due to \eqref{est:omega} and the inequality $(1 + c y \sqrt{|\lambda|}) e^{-c y  \sqrt{|\lambda|} } \leq  1 + \frac{1}{e} = C'$, we see that
\begin{align*}
|e^{-y \omega (\lambda, z)} |
\leq  e^{-c y (|z| + \sqrt{|\lambda|}) }
=e^{-cy\sqrt{|\lambda|} } \cdot e^{-c y |z|  }
\leq \frac{C'}{1+c y \sqrt{|\lambda|}} e^{- \delta y |z| } \, ,
\end{align*}
where we assumed $\delta$ to be small enough. Thanks to an obvious analogue of \eqref{est:frac2}, we obtain
\begin{align*}
|  m_3 (\lambda, z, y) | 
&=\Big|  \frac{1}{\alpha + \lambda + \omega (\lambda, z)} e^{-y \omega (\lambda, z)} \Big|
\leq  \frac{C'}{ | \alpha + \lambda + \omega (\lambda, z)|} \cdot \frac{e^{- \delta y |z| }}{1+c y \sqrt{|\lambda|}}  \\
&\leq  \frac{C}{(1+c y \sqrt{|\lambda|})( \alpha + |\lambda| ) } e^{- \delta y |z| } . 
\end{align*}

Next, using a modification of both \eqref{est:frac2} and \eqref{est:frac3}, we find
\begin{align*}
|  z m_3 (\lambda, z, y) | 
&\leq C' \frac{ |z|}{ | \alpha + \lambda + \omega (\lambda, z)|} \cdot \frac{e^{- \delta y |z| }}{1+c y \sqrt{|\lambda|}}  \\
&\leq  \frac{C}{(1+c y \sqrt{|\lambda|}) \sqrt{\alpha + |\lambda|} } e^{- \delta y |z| } . 
\end{align*}

\underline{Ad $m_2$.} We have
\begin{align*}
 m_2(\lambda, z, y) &= \partial_y  m_0(\lambda, z, y) \\
&= \frac{1}{\lambda} \cdot \frac{z + \omega (\lambda, z)  }{\alpha + \lambda + z + \omega (\lambda, z)} \cdot \frac{\partial }{\partial y}  ( e^{-y \omega (\lambda, z)} - e^{-y z} ) \\
&= \frac{1}{\lambda} \cdot \frac{z + \omega (\lambda, z)  }{\alpha + \lambda + z + \omega (\lambda, z)} \cdot  ( -\omega (\lambda, z) e^{-y \omega (\lambda, z)} + ze^{-y z} \pm z e^{-y \omega (\lambda, z)}) \\
&=  \frac{z + \omega (\lambda, z)  }{\alpha + \lambda + z + \omega (\lambda, z)} \cdot \frac{  z (e^{-y z} - e^{-y \omega (\lambda, z)}) + (z - \omega (\lambda, z)) e^{-y \omega (\lambda, z)} } {\lambda} \\
&= - m_1(\lambda, z, y) - \frac{1}{\alpha + \lambda + z + \omega (\lambda, z)} e^{-y \omega (\lambda, z)} \, .
\end{align*}
We have already shown that the first term is bounded, and the second one can be estimated in the same way as  $m_3(\lambda, z, y)$. Similarly,  
\begin{align*}
z m_2(\lambda, z, y) =  z\partial_y  m_0(\lambda, z, y)
&= -  z m_1(\lambda, z, y) - \frac{z}{\alpha + \lambda + z + \omega (\lambda, z)} e^{-y \omega (\lambda, z)} 
\end{align*}
is bounded as both terms are bounded due to the previous parts. 

In the above, we proved \eqref{est:mi-first} and the first part of \eqref{est:mi-second}; it remains to estimate $\partial_y m_i$. For $i = 1$, we have
\begin{align*}
 \partial_y m_1 (\lambda, z, y) =  z m_2 (\lambda, z, y),
\end{align*}
hence, the bound of this follows from the already proven estimate. Next, for $i = 3$, we get
\begin{align*}
 \partial_y m_3  (\lambda, z, y)
&= -\frac{ \omega (\lambda, z) }{\alpha + \lambda + \omega (\lambda, z)} e^{-y \omega (\lambda, z)} .
\end{align*}
There obviously holds $|\omega (\lambda, z) | \leq C (\sqrt{|\lambda|}+|z|)$. Hence, the boundedness of $ \partial_y m_3 $ follows from already established inequalities for $i = 3$. Finally, for $i = 2$, we can compute that
\begin{align*}
 \partial_y m_2 (\lambda, z, y)
&=  - \partial_y m_1(\lambda, z, y) + \frac{\omega (\lambda, z)}{\alpha + \lambda + z + \omega (\lambda, z)}   e^{-y \omega (\lambda, z)} .
\end{align*}
Invoking \eqref{est:frac3} for $i = 1$, $3$, the estimate for $m_2$
is complete.
\par
Concerning $m_4$, it is enough to invoke \eqref{def-m1-m4} and note
that $|e^{-yz}| = e^{-(\cos\theta)y}$ whenever $z\in \Sigma_\theta$.
\end{proof}

We can now state and prove the following key result.
\begin{lemma}	\label{lem:fundLp}
The solution $\bu$ to \eqref{eq:InsideEquation}--\eqref{eq:BoundaryEquation}
satifies the estimates
\begin{align}	\label{est-fund1}
\norm{\bu}{L^p(\Bdry)}{} + \norm{\bu}{L^p(\Omega)}{}
&\le \frac{K_1}{|\lambda|} \norm{\bm \phi}{L^p(\Bdry)}{} ,
\\				\label{est-fund2}
\norm{\nabla \bu}{L^p(\Omega)}{} 
	&\le \frac{K_2}{\sqrt{|\lambda|}} \norm{\bm \phi}{L^p(\Bdry)}{} .
\end{align}
Here the constants $K_1$ and $K_2$ depend on $p$ and $d$, but
are independent of $\lambda$.
\end{lemma}

\begin{proof}
In view of \eqref{mult:simpler} and \Cref{lem:mult2}, we have
\[
	K_1 \le k(\lambda,0) 
		+ \left( \int_0^\infty k^p(\lambda,y)\,dy \right)^{1/p} ,
\]
where
\[
	k(\lambda,y) = \max_{i=1,2,3} \sup_{z \in \Sigma_\theta}
					|\lambda| m_i(\lambda,z,y) .
\]
By \Cref{thm:Lemma2} we thus obtain
\begin{align*}
	K_1 \le \frac{c |\lambda|}{\alpha + |\lambda|}
		\left[ 1 + \left( \int_0^\infty \frac{{\rm d}y}{(1+y\sqrt{|\lambda|})^p}
			\right)^{1/p} \right]
		\le c \frac{ |\lambda|^{1-1/2p}}{\alpha + |\lambda|} \, .
\end{align*}
This is bounded independently of $\lambda$ in a given range.
Estimate of $K_2$ is obtained similarly.
\end{proof}

In the Hilbert case $p=2$, one obtains the estimate of 
$\bu \in W^{s,2}$ directly.
\begin{lemma}	\label{lem:fundH2}
The solution $\bu$ to \eqref{eq:InsideEquation}--\eqref{eq:BoundaryEquation} for $\lambda > 0$
satifies the estimates
\begin{equation}	\label{est-H2}
c (\lambda)
\norm{\bm \phi}{W^{s-1/2,2}(\Gamma)}{}
\le
\norm{\nabla \bu}{W^{s,2}(\Omega)}{}
\le C (\lambda)
\norm{\bm \phi}{W^{s-1/2,2}(\Gamma)}{}
\end{equation}
for any $s \in[1/2 , 3/2)$. In particular, $\bu \in W^{s+1,2}(\Omega)$
if (and only if) $\bm\phi \in W^{s-1/2,2}(\Gamma)$.
\end{lemma}
\begin{proof}
Let us comment first on the last part of the statement. If $\bm\phi \in W^{s-1/2,2}(\Gamma)$, then $\bu \in W^{s+1,2}(\Omega)$ follows from \eqref{est-H2} together with the energy inequality; see \eqref{res:weak}. Note that $\lambda > 0$ is needed since the Poincaré inequality in the half-space is not valid. On the other hand, if $\bu \in W^{s+1,2}(\Omega)$, then $\bm\phi \in W^{s-1/2,2}(\Gamma)$ holds due to \eqref{eq:BoundaryEquation} and the trace theorem.

We will sketch the proof of \eqref{est-H2} for $d=2$, the general case being similar.
The key point is that since $p=2$, we can use the Fourier transform to
characterize the Sobolev norm of $\phi$ as 
\[
	\norm{\phi}{W^{s-1/2,2}(\Gamma)}{2}
	= \int_\rr (1+|\xi|^2)^{s-1/2}|\widehat{\phi}(\xi)|^2\, {\rm d}\xi \,,
\]
which we want to relate to the velocity gradient norm
\[
	\norm{\nabla \bu}{W^{s,2}(\Omega)}{2}
	=
	\int_{\rr^2} (1 + \xi^2 + \eta^2)^{s}
			| \mathcal{F} (\nabla\bu)(\xi,\eta)|^2 \, {\rm d}\xi  {\rm d}\eta \,,
\]
where $\mathcal{F}$ denotes the full Fourier transform in $\mathbb{R}^2$ of a function that is understood to be evenly extended with respect to the variable $y$. We will also frequently use the symbol $A\simeq B$ as a shorthand
for $A \le c_1 B$ and $B \le c_2 A$, where $c_1$, $c_2>0$ are some
irrelevant constants, and $A \lesssim B$ for a one-sided estimate.
\par
Since $d=2$, we have $\bu = (u_1,u_2)$, and relations
\eqref{def-m0}, \eqref{mult:simpler} lead to formulas for
the partial Fourier transform of $\widehat{\nabla}=(i\xi,\partial_y)$
\begin{align*}
\widehat{\nabla u_1}(\xi,y) &= 
-\big(i\xi \partial_y m_0(z,y) , \partial^2_y m_0(z,y) \big) 
	\widehat{\phi}(\xi) ,
\\
\widehat{\nabla u_2}(\xi,y) &=
\big( i\xi z m_0(z,y) , z \partial_y m_0(z,y) \big) \frac{i\xi}{|\xi|}
	    \widehat{\phi}(\xi) ,
\end{align*}
where $z = |\xi|$ and
\[
	m_0(z,y) = R(z) \big( e^{-\omega(z)y} - e^{-zy} \big) .
\]
Recall that $\omega(z) = \sqrt{\lambda + z^2}$; for simplicity, 
we omit the dependence on  $\lambda$, which is now fixed. 
Note that likewise $R(z)$ is not relevant as $|R(z)| \simeq 1$.
\par
It will be useful to set 
\begin{align*}
Z_j(z,\eta) = \mathcal{F}_{y \to \eta} [ \partial^j_y ( e^{-\omega(z)y} - e^{-zy}) ], \quad j=0,1,2 ;
\end{align*}
even extension with respect to $y$ is used before applying
the transform. The full Fourier transform of $\nabla \bu$ now consists of
\begin{align*}
\mathcal{F}(\nabla u_1)(\xi,\eta) &=
-\big( i\xi  Z_1 (z,\eta) ,  Z_2(z,\eta) \big)
	R(z) \widehat{\phi}(\xi)  , 
\\
\mathcal{F}(\nabla u_2)(\xi,\eta) &=
\big( i\xi z  Z_0(z,\eta) , z  Z_1(z,\eta) \big)
	R(z) \widehat{\phi}(\xi)  .
\end{align*}
The key step is to evaluate and estimate 
$Z_j(z,\eta) $. In view of the formula
\begin{equation*}	
	\mathcal{F}_{y \to \eta}[ \exp(-a|y|) ]
	= \frac{a}{a^2+\eta^2} \, , \quad \mathfrak{Re}\, a > 0 ,
\end{equation*}
we obtain
\begin{align*}
Z_0(z,\eta)
&= \frac{\omega(z)}{\omega^2(z)+\eta^2}
- \frac{z}{z^2+\eta^2} = \frac{ \lambda }{\omega(z)+z} \cdot
	\frac{ \eta^2 - z \omega(z) }{(\omega^2(z)+\eta^2)(z^2+\eta^2)} \, ,
\\
Z_1(z,\eta)
&= 
\frac{z^2}{z^2+\eta^2} - \frac{\omega^2(z)}{\omega^2(z)+\eta^2}
=
\frac{ -\lambda \eta^2 }{(\omega^2(z)+\eta^2)(z^2+\eta^2)} \, ,
\\
Z_2(z,\eta)
 	&= \frac{\omega^3(z)}{\omega^2(z)+\eta^2}
	- \frac{ z^3 }{z^2+\eta^2} 
\\
&= \frac{\lambda( \eta^2(z^2 + z\omega(z) + \omega^2(z)) + z^2 \omega^2(z)))}%
{(z+\omega(z))(\omega^2(z)+\eta^2)(z^2+\eta^2)} \, .
\end{align*}
Recalling now that $\lambda > 0$ is fixed, it is straighforward to
deduce that $\omega(z)\simeq 1+z$, $\omega^2(z) \simeq 1 + z^2$, and
hence also $|\eta^2 - z\omega(z)| \lesssim \eta^2 + z^2 + z$, and
$|\eta^2(z^2 + z\omega(z) + \omega^2(z)) + z^2 \omega^2(z)|
\simeq (\eta^2+z^2)(1+z^2)$. We thus can write
\begin{align*}
|\mathcal{F} (\nabla u_1)(\xi,\eta)| &\simeq 
\big( |z Z_1(z,\eta)| + |Z_2(z,\eta)| \big) 
	|\widehat{\phi}(\xi)| \simeq \frac{1+z}{1+z^2+\eta^2} |\widehat{\phi}(\xi)| ,
	\\
|\mathcal{F}(\nabla u_2)(\xi,\eta)| &\simeq 
\big( |z^2 Z_0(z,\eta)| + |z Z_1(z,\eta)| \big) 
	|\widehat{\phi}(\xi)| \lesssim \frac{1+z}{1+z^2+\eta^2} |\widehat{\phi}(\xi)| .
\end{align*}
This eventually leads to a key pointwise estimate
\begin{align*}
|\mathcal{F} (\nabla \bu)(\xi,\eta)|	\simeq
\frac{(1+|\xi|^2)^{1/2}}{1+|\xi|^2+\eta^2} |\widehat{\phi}(\xi)| .
\end{align*}
It is now straighforward to deduce
\begin{align*}
	\norm{\nabla \bu}{W^{s,2}(\Omega)}{2} 
&= \int_{\rr^2} (1+|\xi|^2+\eta^2)^s |\mathcal{F} (\nabla \bu )(\xi,\eta)|^2 \, {\rm d}\xi  {\rm d}\eta
\\
		&\simeq \int_{\mathbb{R}} (1+|\xi|^2) \Big( \int_{\mathbb{R}}  (1+|\xi|^2+\eta^2)^{s-2} \, {\rm d}\eta \Big) |\widehat{\phi}(\xi)|^2 \, {\rm d}\xi 
\\
		&\simeq \int_{\mathbb{R}} (1+|\xi|^2)^{s-1/2} |\widehat{\phi}(\xi)|^2 \, {\rm d}\xi 
		= \norm{\phi}{W^{s-1/2}(\Gamma)}{2} .
\end{align*}
Here we have also used that
\begin{equation*}	
	\int_{\rr} \frac{{\rm d}\eta}{(b^2+\eta^2)^\sigma} 
		\simeq b^{1-2\sigma}, \quad \sigma > 1/2 . \qedhere
\end{equation*}
\end{proof}

\section{Proof of the main results} 	\label{S:5}

\begin{proof}[Proof of \autoref{thm:Main1}]

We start with formally recovering the resolvent estimate \eqref{est:Res}. 
Decompose the 
solution as $\bu = \bu^1 + \bu^2$, where $\bu^1$ solves the same problem, but with the zero Dirichlet boundary condition, and $\bu^2$ has the zero right-hand side in $\Omega$ and takes care of the boundary condition. To be specific; the function $\bu^1$ solves the system
\begin{equation}	\label{Split1}
\begin{aligned}
(\lambda - \Delta) \bu^1 + \nabla \pi^1 &= \bm f \quad \text{ in } \Omega ,  \\
\diver \bu^1 &= 0 \quad  \text{ in } \Omega ,  \\
\bu^1  &= 0 \quad  \text{ on } \Gamma  ,
\end{aligned}
\end{equation}
and $\bu^2$ solves 
\begin{equation}	\label{Split2}
\begin{aligned}
(\lambda - \Delta) \bu^2 + \nabla \pi^2 &= 0 \quad \text{ in } \Omega ,  \\
\diver \bu^2 &= 0 \quad  \text{ in } \Omega , \\
\bu^2 \cdot \bm n &= 0 \quad  \text{ on } \Gamma ,  \\
(\lambda + \alpha)\bu^2 + 2 [(\Du^2) \bm n]_{\tau} &= \bm \Phi \quad  \text{ on } \Gamma  ,
\end{aligned}
\end{equation}
where
\begin{equation}	\label{Split3}
\bm \Phi = \bm h - 2 [(\Du^1) \bm n]_{\tau} .
\end{equation}
Assume that $\lambda \in \mathbb{C}$ with $|\lambda| > \omega$, 
where $\omega > 0$ is fixed.  Due to the standard theory 
(see e.g. \cite[Theorem 1.2]{FaSo94}) we know that
\begin{align}	\label{est:U1}
  || \lambda \bu^1 ||_{L^p(\Omega)} + || \nabla^2 \bu^1 ||_{L^p(\Omega)}  
\leq c || \bm f ||_{L^p(\Omega)} \,,
\end{align}
and in particular, 
\begin{align*}
  || \bu^1 ||_{L^p(\Omega)}   \leq \frac{c}{|\lambda|} 
 || \bm f ||_{L^p(\Omega)} \,.
\end{align*}
Because of the trace theorem we get 
\begin{align*}
  || [(\Du^1) \bm n]_{\tau} ||_{L^p(\Gamma)} \leq c  || \bm f ||_{L^p(\Omega)}  ,
\end{align*}
and thus, there also holds   
\begin{align*}
|| \bm  \Phi  ||_{L^p(\Gamma)}  & =   ||\bm h - 2 [(\Du^1) \bm n]_{\tau} ||_{L^p(\Gamma)}
   \leq c || (\bm f, \bm h) ||_{{L^p(\Omega)} \times {L^p(\Gamma)}}  . 
\end{align*}
Note that $c$ remains bounded as $|\lambda| \to \infty$.
Hence, it suffices to show 
\begin{align*}
  || (\bu^2, \text{tr}\, \bu^2 ) ||_{{L^p(\Omega)} \times {L^p(\Gamma)}}    \leq \frac{c}{|\lambda|}   || \bm \Phi  ||_{L^p(\Gamma)} .  
\end{align*}
But this follows by \autoref{lem:fundLp} above, 
in particular the estimate \eqref{est-fund1}, because $\bm \Phi = (\bm \phi, 0)$.

\par
So far, the argument was formal. Using a suitable approximation,
one shows that the resolvent problem \eqref{Res:inside}--\eqref{Res:bound}
has at least one solution. In view of \eqref{est:U1} and the second part
of \autoref{lem:fundLp}, i.e. \eqref{est-fund2}, we see that $\bu \in W^{1,p}(\Omega)$.
More precisely, we have 
\begin{equation}	\label{est:Res-w1}
\norm{\bu}{W^{1,p}(\Omega)}{} 
\le C \big(
\norm{\bef}{L^p(\Omega)}{} + \norm{\beh}{L^p(\Bdry)}{}
	\big),
\end{equation}
where $C$ depends on $\lambda$. Let us show uniqueness; clearly,
it suffices to consider $\bu$ a solution to the homogeneous problem.
We set  $\widetilde \bef = |\bu|^{p-2}\overline{\bu}$,
$\widetilde \beh = |\bu|^{p-2}\overline{\bu}$ and observe that 
$\widetilde \bef \in L^{p'}(\Omega)$, $\widetilde \beh \in L^{p'}(\Bdry)$.
By the above, there is some $\widetilde{\bu} \in W^{1,p'}_{\sigma,n}(\Omega)$,
a solution corresponding
to these data. Testing the weak formulation for $\bu$ by $\widetilde{\bu}$
and vice versa, one deduces that
\[
	0 = \int_\Omega \widetilde{\bef}\cdot \bu \,{\rm d}x
	+ \int_\Bdry \widetilde{\beh} \cdot \bu \,{\rm d}S
	= \norm{\bu}{L^p(\Omega)}{p} + \norm{\bu}{L^p(\Bdry)}{p} \,.
\]
Hence $\bu =0$.
\end{proof}	

\begin{proof}[Proof of \autoref{thm:Main2}]

We split the solution $\bu = \bu^1 + \bu^2$ as above, 
allowing moreover $\diver \bu^1=g$ in \eqref{Split1}.
Hence, part (iii) of the Theorem is taken into account.

Let us now consider the strong estimates, i.e. part (ii).
Regularity of $\bu^1$ follows as in \cite{FaSo94}, see also
\cite[Theorem IV.3.2]{Galdi11}.  Note that $\bm \Phi 
\in B^{1-1/p,p}_p(\Bdry)$ by the trace theorem \cite[Theorem 7.39]{AF03}.
\par
Concerning $\bu^2$, we note that for a flat boundary, we have
the so-called Robin boundary condition, whence the $W^{2,p}$-regularity 
follows by \cite{ShiShi07} or \cite{kaji22}. 

\par
Nevertheless, 
we will sketch the proof that $\nabla \pi \in L^p(\Omega)$.
In particular, we can adapt the argument from \cite{FaSo94},
which is here simplified in view of \autoref{lem:mult2}.
(The estimate $\nabla \bu^2$ is done similarly, or just
follows by the regularity for the laplacian again.)
\par
Introduce an auxiliary function $\bv = (\bv', 0)$ as the solution to
\begin{align*}
(-\delta^2 \Delta' - \partial_d^2) \bv' &= 0 \quad \text{ in } \Omega , \\
-\partial_d \bv' &= \bm \phi \quad \text{ on } \Gamma ,
\end{align*}
where $\delta>0$ is a suitably small number.
We know (see e.g. \cite[Theorem 14.1]{ADN59}) that 
\begin{align}	\label{reg-vv}
\norm{ \nabla^2 \bv }{L^p(\Omega)}{}
\le c \norm{ \bm \phi}{B^{1-1/p,p}_p(\Gamma)}{}
\,.
\end{align}
We proceed by noting that
\begin{align*}
\widehat{\nabla' \bv}(\xi, x_d) = \frac{i\xi}{\delta |\xi|} e^{-\delta |\xi| x_d} \widehat{\bm \phi}(\xi) , \quad \textrm{or} \qquad
\frac{i\xi}{|\xi|}\cdot \widehat{\bm\phi}(\xi) = 
	\delta e^{\delta |\xi| x_d} \widehat{\nabla' \bv}(\xi,x_d) ,
\end{align*}
and hence, in view of \eqref{mult:simpler}, we can write
\begin{equation*}	
\widehat{\pi}(\xi,y) = f(z)e^{-\eta zy} \widehat{\nabla'\bv}(\xi,y) .
\end{equation*}
Here $\eta=1-\delta$, $z=|\xi|$, $y=x_d$ and 
\[
	f(z) = 
\frac{ \omega(\lambda,z) + z }{\alpha + \lambda + \omega(\lambda,z) + z}
\]
is bounded and holomorphic in $\Sigma_{\pi/2}$. Hence
\begin{align*}
	\widehat{\nabla\pi}(\xi,y) 
	&= (i\xi,\partial_y) \big[ f(z) e^{-\eta zy} \widehat{\nabla'\bv}(\xi,y) \big]
\\
	&= \big[ f(z)e^{-\eta zy} \big] \big( 
		\widehat{(\nabla')^2 \bv }(\xi,y) , 
		(-\eta z + \partial_y) \widehat{\nabla' \bv}(\xi,y) \big) .
\end{align*}
The trick is now to rewrite the last term as
\begin{equation*}
	(-\eta z + \partial_y) \widehat{\nabla' \bv}(\xi,y) = 
	-\frac{i\eta \xi}{|\xi|}\cdot \widehat{(\nabla')^2 \bv }(\xi,y)
		+ \widehat{\partial_y \nabla'\bv}(\xi,y) .
\end{equation*}
We conclude that $\widehat{\nabla \pi} \simeq m_5(z,y) \widehat{\nabla^2 \bv}$,
where $m_5(z,y)$ is holomorphic and bounded in $z$ uniformly with respect to $y\ge0$.
By the second part of \autoref{lem:mult2}, cf. \eqref{est-ki}, we deduce
\begin{equation*}
		\norm{\nabla \pi}{L^p(\Omega)}{} \le c
		\norm{\nabla^2 \bv}{L^p(\Omega)}{} .
\end{equation*}
In view of \eqref{reg-vv}, the conclusion follows.

\par\medskip

It remains to show part (i). The existence and regularity of $\bu^1$
follows from \cite[Theorem IV.3.3]{Galdi11}.
The key step is to bound $\Du^1 \in B^{-1/p,p}_p(\Bdry)$, 
to which we need the following auxiliary result.

\begin{lemma}	\label{lm:ntrace}
Let $\bm U \in L^p(\Omega)$ and $\diver \bm U \in W^{-1,p}(\Omega)$.
Then $(\bm  U\cdot \bm n)_\tau \in B^{-1/p,p}_{p;n}(\Bdry)$.
\end{lemma}
\begin{proof}
Integrating by parts, we write
\[
\int_\Bdry (\bm U \cdot \bm n)\bm \varphi \,{\rm d}S	
= \int_\Omega \bm U \cdot \nabla \bm \varphi \,{\rm d}x 	
	+ \int_\Omega (\diver \bm U)\bm \varphi \,{\rm d}x 	
\]
and note that the right-hand side can be bounded in terms
of $\bm \varphi \in W^{1,p'}_n(\Omega)$. Denoting formally
by $\gamma$ the trace operator, it remains to observe that 
\[ 
\gamma (W^{1,p'}_n(\Omega)) = B^{1/p,p'}_{p';n} (\Bdry) = 
(B^{-1/p,p}_{p;n}(\Bdry))^* \,. \qedhere
\] 
\end{proof}
Applying now the above lemma to $\bm U = \nabla \bu^1$, we see
that $\bu^2$ solves \eqref{Split2}--\eqref{Split3}
where $\bm \Phi \in B^{-1/p,p}_{p;n}(\Bdry)$.
\par
Set $\bv = \bpjx \bu^2$, where $\bpjx$ is the
Bessel potential applied in the tangential directions $x'$.
Obviously, $\bv$ solves \eqref{Split2} with the boundary 
data $ \bpjx \bm \phi \in B^{1-1/p,p}_{p}$.
Hence $\bv \in W^{2,p}(\Omega)$, and thus 
$\bu^2 \in W^{1,p}(\Omega)$ as required.
\end{proof}

{\bf Acknowledgment.} The authors are grateful to anonymous
referee for his/her careful reading of the manuscript.

\bibliographystyle{plain}
\bibliography{bibliography}

\end{document}